\documentclass[a4paper,12pt]{article}
\usepackage{amsmath}
\usepackage{amssymb}
\usepackage{amsthm}
\newtheorem{Thm}{Theorem}[section]

\newtheorem{Lem}[Thm]{Lemma}
\newtheorem{Prop}[Thm]{Proposition}

\theoremstyle{definition}
\newtheorem{Def}[Thm]{Definition}

\begin{document}

\title{On the range of self-interacting random walks on an integer interval}
\author{Kazuki Okamura}
\date{}
\maketitle

\begin{abstract}
We consider the range of a one-parameter family of self-interacting  walks on the integers up to the time of exit from an interval.
We derive the weak convergence of an appropriately scaled range. 
We show that the distribution functions of the limits of the scaled range satisfy a certain class of de Rham's functional equations. 
We examine the regularity of the limits.  
\end{abstract}

\section{Introduction}

The range of random walk has been studied for a long time.
Examining the range at the time the random walk leaves an interval is a simple and natural concern.
Recently, Athreya, Sethuraman and  T\'{o}th \cite{AST} considered questions of this kind. 
They studied the range, local times and periodicity or ``parity" statistics of some nearest-neighbor Markov random walks up to the time of exit from an interval of $N$ sites.
They derived several associated scaling limits as $N \to \infty$ and  related the limits to various notions such as the entropy of an exit distribution, generalized Ray-Knight constructions, and Bessel and Ornstein-Uhlenbeck square processes. 

Inspired by \cite{AST}, we consider the ranges of a certain class of  self-interacting random walks up to the time of exit from an interval.
The study of  self-interacting walks originated from the modeling of polymer chains in chemical physics.
There are various models in this study.
We consider the model defined by Denker and Hattori \cite{DH}, Hambly, Hattori and Hattori \cite{HHH}, Hattori and Hattori \cite{HH1}, \cite{HH2}. 
They constructed  a natural one-parameter family of self-repelling and self-attracting walks on  $\mathbb{Z}$ and the infinite pre-Sierpi\'{n}ski gasket.
It interpolates continuously between self-avoiding walk and simple random walk  in the sense of exponents.

In general, most of the studies of  self-interacting walks are difficult  due to the lack of  Markov property, even if they are one-dimensional.
In the studies of  Markov walks,  we can use techniques in analysis, especially, potential theory. 
However, in the case of non-Markov walks, we cannot use most of the techniques used in the studies of  Markov walks.
Most of the arguments in \cite{AST} depend heavily on the Markov property.
Therefore, we have to use  alternative methods for our study.
We apply a recent result by the author \cite{O2} which considers a certain class of de Rham's functional equations. 

Now we state our settings and results briefly.
Let $W_{\infty}$ be the path space of the nearest-neighbor walk starting at $0$ on $\mathbb{Z}$.
Let $\{P^{u}\}_{u \geq 0}$ be a one-parameter family of probability measures on $W_{\infty}$ defined by  \cite{DH} and \cite{HH2}. 
We will give the precise definitions of them in Section 2.
$P^{0}$ defines the self-avoiding walk  on $\mathbb{Z}$ and $P^{1}$ defines the standard simple random walk.
If $u \ne 1$, $P^{u}$ defines a non-Markov random walk on $\mathbb{Z}$.

\begin{Def}
Let $n \in \mathbb{N} = \{1,2, \dots \}$ and $\omega \in W_{\infty}$. 
Let $R_{n}(\omega)$ be the range of $\omega$ up to the time of exit from $\{-2^{n}, \dots, 2^{n}\}$, that is, 
\[ R_{n}(\omega) = \left(\textrm{the number of points which} \  \omega  \ \textrm{visits before it hits the points} \,  \{\pm2^{n}\}\right).\]
Note that $2^{n} \le R_{n} \le 2^{n+1}-1$.  
\end{Def}

Then, we have the following results which are analogous to  \cite{AST}, Proposition 2.1.

\begin{Thm}
$(1)$ Let $u \ge 0$.
Then, the random variables $\{(R_{n}/2^{n}) - 1\}_{n}$ converges weakly to 
a distribution function $f_{u}$ on $[0,1]$, $n \to \infty$.\\ 
$(2)$ Let $u > 0$.
Then 
$f_{u}$ satisfies  a certain class of de Rham's functional equations $\cite{de}$ $:$ 
\begin{equation}
	f(x) = \begin{cases}
						\Phi(A_{u,0} ; f(2x)) & 0 \leq x \leq 1/2 \\
						\Phi(A_{u,1} ; f(2x-1)) & 1/2  \leq x \leq 1, \tag{1.1}
				          \end{cases}
\end{equation}
where we let 
\[ \Phi(A ; z) = \frac{az + b}{cz + d} \, \text{ for } A = \begin{pmatrix} a & b  \\ c & d \end{pmatrix}, \text{ and, }\]
\[ A_{u, 0} = \begin{pmatrix} x_{u} & 0  \\ -u^{2}x_{u}^{2} & 1 \end{pmatrix}, \,  \, A_{u, 1} = \begin{pmatrix} 0 & x_{u} \\ -u^{2}x_{u}^{2} & 1 - u^{2}x_{u}^{2} \end{pmatrix},  \, \,  x_{u} = \frac{2}{1+\sqrt{1+8u^{2}}}. \] \\
$(3)$ 
Let $\widetilde P^{u}$ be the probability measure on $[0,1]$ such that its distribution function is $f_{u}$. 
If $u = 1$, $\widetilde P^{u}$ is absolutely continuous with respect to the Lebesgue measure on $[0,1]$.  
If $u \neq 1$, $\widetilde P^{u}$ is singular.
\end{Thm}
We remark that $\widetilde P^{0} = \widetilde P^{0}_{n} = \delta_{\{0\}}$, where $\delta$ denotes a point mass.

Let us denote the Hausdorff dimension of $K \subset [0,1]$ 
by $\dim_{H}(K)$.
Let us define the Hausdorff dimension of a probability measure $\mu$ on $[0,1]$ by $\dim_{H} \mu = \inf \{\dim_{H}(K) : K \in \mathcal{B}([0,1]), \mu(K) = 1\}$.
Let $s(p) = -p\log p - (1-p) \log (1-p)$ for $p \in [0,1]$.

If $0 < u < \sqrt{3}$, $(A_{u,0}, A_{u,1})$ satisfies the conditions (A1) - (A3) in \cite{O2}, 
so we can apply the results in \cite{O2} to this case 
and obtain the following results. 
We refer the reader to  \cite{O2} for details. 

\begin{Thm}
$(1)$ If $u \neq 1$ and $0 < u < \sqrt{3}$,
then $\dim_{H} \widetilde P^{u} < 1$.\\
$(2)$ If $0 < u < 1$, 
then $\dim_{H} \widetilde P^{u} \leq s(x_{u})/\log 2$.
Moreover, $\widetilde P^{u}(K) = 0$ for any Borel set $K$ with $\dim_{H}(K) < s(2x_{u}/(1+x_{u}))/ \log 2$.
\end{Thm}

We also examine whether $\widetilde P^{u}$ has atoms.

\begin{Thm}
$(1)$ Let $u \leq \sqrt{3}$.
Then, $\widetilde P^{u}$ has no atoms.\\
$(2)$ Let $u > \sqrt{3}$.
Then, $\widetilde P^{u}(\{x\}) > 0$ for any $x \in D \cap (0,1]$.
Here $D$ is the set of dyadic rationals on $[0,1]$. 
\end{Thm}

In Section 2, we will describe the settings.
In Section 3, we will show Theorem 1.2 and Theorem 1.4.

\textit{Acknowledgements.} \, The author wishes to express his gratitude to Professor Shigeo Kusuoka and  Professor Tetsuya Hattori and Professor \\
Kumiko Hattori for their comments. 
The author wishes to express his gratitude to the referee for careful reading.  

\section{Preliminaries}

We briefly state our settings by following \cite{DH} and \cite{HH1}. 
See the references for details. 

For each $n \in \mathbb{N} \cup \{0\}$, let
\[ W(n) = \{  (\omega(0),\omega(1) \dots \omega(n)) \in \mathbb{Z}^{n+1} : \omega(0) = 0, |\omega(i) - \omega(i + 1)| = 1, \ 0 \leq i \leq n-1 \}.\]
Let $W^{*} = \cup_{n = 0}^\infty W(n)$.
Let $L(\omega) = n$ for $\omega \in W(n)$.
For $\omega \in W^{*}$, we define $T_{i}^{M}(\omega)$, 
$i, M \in \mathbb{N} \cup \{0\}$,  
by $T_{0}^{M}(\omega) = 0$,
\[ T_{i}^{M}(\omega) = \min \left\{ j > T_{i - 1}^{M}(\omega) : \omega(j) \in {2^{M}\mathbb{Z}} \setminus \{\omega(T_{i - 1}^{M}(\omega))\} \right\}, i \ge 1.\]
Let $T_{i}^{M}(\omega) = +\infty$ if the above minimum does not exist.

We define a decimation map  $Q_{M} : W^{*} \to W^{*}$, $M \in \mathbb{N}$,  by $(Q_{M}\omega)(i) = \omega(T_{i}^{M}(\omega))$ 
for $i$ such that $T_{i}^{M}(\omega) < +\infty$.
Let $Q_{0}$ be the identity map on $W^{*}$. 
Let $(2^{-M}Q_{M}\omega)(i) = 2^{-M}\omega(T_{i}^{M}(\omega))$.
Then, $2^{-M}Q_{M}\omega \in W^{*}$ and $L(2^{-M}Q_{M}\omega) = k$,
where $k = \max \{i: T_{i}^{M}(\omega) < \infty \}$.
Let $W_{N, +(resp. -)} = \{ \omega \in W^{*} : L(\omega) = T_{1}^{N}(\omega), \omega(T_{1}^{N}(\omega)) = + (\textrm{resp.} -) 2^{N} \}$ and 
$W_{N} = W_{N,+} \cup W_{N,-}$.

For $\omega \in W_{N + n,+}$, let $\omega^{\prime} = 2^{-N}Q_{N}\omega$.
For $1 \leq j \leq L(\omega^{\prime})$, we let 
$\omega_{j} = \left(0, \, \omega(T_{j - 1}^{N}(\omega) + 1) - \omega(T_{j - 1}^{N}(\omega)), \dots , \, \omega(T_{j }^{N}(\omega)) - \omega(T_{j - 1}^{N}(\omega))\right) \in W_{N}$,
and, 
$\widetilde\omega_{j} = \textrm{sign}\left(\omega(T_{j }^{N}(\omega)) - \omega(T_{j - 1}^{N}(\omega))\right)\omega_{j} \in W_{N,+}$.

Now we will define  a probability measure $P_{N, \pm}^{u}$, $u \geq 0$, on $W_{N, \pm}$ by induction on $N$ in the following manner.
We recall that $x_{u} = 2 / (1+\sqrt{1+8u^{2}})$. 
Let $P_{1,+}^{u}(\{ \omega \}) = u^{L(\omega) - 2}x_{u}^{L(\omega) - 1}$, $\omega \in W_{1,+}$,
where we adopt the conventions $0^{0} = 1$ and  $0^{n} = 0$, $n \geq 1$.
For $\omega \in W_{N + 1,+}$, 
let \[ P_{N + 1, +}^{u}(\{\omega\}) = P_{1, +}^{u}(\{\omega^{\prime}\}) \prod_{i = 1}^{L(\omega^{\prime})} P_{N, +}^{u}(\{\widetilde\omega_{i}\}). \tag{2.1} \]
We define $P_{N, -}^{u}(\{\omega\}) = P_{N, +}^{u}(\{-\omega\})$ for $\omega \in W_{N, -}$, $N \in \mathbb{N}$.
Let $P_{N}^{u}$ be a probability measure  on $W_{N}$ given by $P_{N}^{u} = (P_{N, +}^{u} + P_{N, -}^{u}) / 2$.

We denote the set of the paths of infinite length by   
\[ W_{\infty} = \left\{ (\omega(0),\omega(1), \dots) \in \mathbb{Z}^{\mathbb{N} \cup \{0\}}  : \omega(0) = 0, \,  |\omega(i) - \omega(i + 1)| = 1, \, i  \geq 0  \right\}.\]
Let the $\sigma$-algebra on this set be the family of subsets which is generated by cylinder sets.
By \cite{DH}, Proposition 2.5, 
there exists a probability measure $P^{u}$ on $W_{\infty}$ 
such that 
\[ P^{u}\left(\left\{ \omega \in W_{\infty} : \omega(j) = \widetilde\omega(j), \,\,0 \leq j \leq L(\widetilde\omega)\right\}\right) = \frac{1}{2}P_{N, + (\textrm{resp.}-)}^{u}(\{\widetilde\omega\}),\]
for any $\widetilde\omega \in W_{N, +  (\textrm{resp.}-)}$, $N \ge 1$.

\section{Range of random walk on the interval $[-2^{n}, 2^{n}]$ and its scaling limit}

Here and henceforth, we assume that $u > 0$.

First we will show Theorem 1.2. 
The main ingredient of the proof is to show that $g_{u}(k/2^{n}) := P_{n, +}^{u}(R_{n} \le 2^{n} + k-1)$ satisfies (1.1) on the dyadic rationals. 
This depends heavily on the definition of $P_{n, +}^{u}$ in Section 2. 
Then, we will see that the right continuous modification of $g_{u}$ satisfies (1.1) on $[0,1]$.
Next, we will show that the distribution of $R_{n}/2^{n} - 1$ converges to $g_{u}$ weakly as $n \to \infty$ 
and examine the regularity of $g_{u}$.

We remark that $P^{u}(R_{n} = 2^{n} + k) = P_{n, +}^{u}(R_{n} = 2^{n} + k)$, $0 \leq k \leq 2^{n}$, $n \geq 1$.

\begin{Lem}
\[ P_{N, +}^{u}\left(\frac{R_{N}}{2^{N}} - 1 \geq \frac{k}{2^{n}} \right)  = P_{n, +}^{u}\left(\frac{R_{n}}{2^{n}} - 1 \geq \frac{k}{2^{n}} \right), \]
for any $N \geq n$, $0 \leq k \leq 2^{n}$ and $n \geq 1$. 
\end{Lem}

\begin{proof}
Let $N > n$. Then, 
\begin{align*}
P_{N, +}^{u}\left(\frac{R_{N}}{2^{N}} - 1 \geq \frac{k}{2^{n}} \right)  &= P_{N, +}^{u} \left(\left\{\omega \in W_{N, +}: \omega \, \textrm{hits the point} \, \{-2^{N - n}k\}\right\}\right)  \\
														  			 &= P_{N, +}^{u} \left(\left\{\omega : Q_{N - n} \omega \, \textrm{hits the point} \,   \{-2^{N - n}k\}\right\}\right)  \\
																	 &= P_{N, +}^{u} \left(\left\{\omega : 2^{-(N - n)}Q_{N - n} \omega \, \textrm{hits the point} \, \{-k\}\right\}\right)  \\
														  			 &= P_{n, +}^{u} \left(\left\{\zeta \in W_{n,+}: \zeta \, \textrm{hits the point} \, \{-k\}\right\}\right)  \\
														   			&= P_{n, +}^{u}\left(\frac{R_{n}}{2^{n}} - 1 \geq \frac{k}{2^{n}} \right),
\end{align*}
where in the fourth equality we have used \cite{DH} Proposition 2.2.
\end{proof}

\begin{Def}
(1) Let $g_{u}$ be a function on $D$ given by 
$g_{u}((k+1)/2^{n}) = P_{n, +}^{u}(R_{n} \le 2^{n} + k)$, 
$-1  \le k \le 2^{n}-1$. 
By Lemma 3.1, this is well-defined.
We immediately see that $g_{u}(x)$ is increasing and $g_{u}(0) = 0, g_{u}(1) = 1$.\\
(2)  Let $\tilde g_{u}$ be a function on $[0,1]$ given by $\tilde g_{u}(x) = \lim_{y \in D, y > x, y \to x} g_{u}(y)$, $0 \leq x < 1$ and $\tilde g_{u}(1) = 1$. This is right continuous.
\end{Def}

The following is a key proposition. 
\begin{Prop}
The function $g_{u}$ satisfies $(1.1)$ on $D$, that is, 
\begin{equation*}
	P_{n + 1,+}^{u}({R_{n + 1}} \leq  {2^{n + 1}} + k)  = 
	\begin{cases}
						\Phi\left(A_{u,0} ; P_{n,+}^{u}(R_{n} \leq  2^{n} + k)\right) & -1 \leq k \leq 2^{n} - 1 \\
						\Phi\left(A_{u,1} ; P_{n,+}^{u}(R_{n} \leq  k)\right) & 2^{n}-1 \leq k \leq 2^{n + 1} - 1.
				          \end{cases}
\end{equation*}
\end{Prop}

\begin{proof}
If $k = -1$, we have that $\Phi\left(A_{u,0} ; P_{n,+}^{u}(R_{n} \leq  2^{n} + k)\right) = \Phi(A_{u,0} ; 0) = 0 = P_{n+1,+}^{u}(R_{n+1} \leq  2^{n+1} + k)$.
If $k = 2^{n}-1$, we have that \\
$\Phi\left(A_{u,0} ; P_{n,+}^{u}(R_{n} \leq  2^{n} + k)\right) = \Phi(A_{u,0} ; 1) = \Phi(A_{u,1} ; 0) = \Phi\left(A_{u,1} ; P_{n,+}^{u}(R_{n} \leq  k)\right)$.
Then, it is sufficient to show this assertion in the following two cases.
For any $\omega \in W_{n+1, +}$, define $(\omega^{\prime}, \widetilde\omega_{1}, \dots, \widetilde\omega_{L(\omega^{\prime})})$ as in Section 2.

Case 1. $0 \leq k \leq 2^{n} - 1$.
We have  
\[  P_{n + 1,+}^{u}(R_{n + 1} \leq  2^{n + 1} + k) = \sum_{m = 1}^{\infty} P_{n + 1, +}^{u}\left(\left\{\omega : L(\omega^{\prime}) = 2m, \, R_{n + 1} (\omega) \leq  2^{n + 1} + k\right\} \right). \]
Since $0 \leq k \leq 2^{n} - 1$, 
we see that $\omega^{\prime} \in W_{1,+}$ does not hit $-1$ for any $\omega \in W_{n+1, +}$ with $R_{n + 1} (\omega) \leq 2^{n + 1} + k$.
Then we see that
\[ \left\{\omega : L(\omega^{\prime}) = 2m, \, R_{n + 1}(\omega) \leq  2^{n + 1} + k \right\} \]
\[ = \left\{\omega : \omega^{\prime} = (0,1,0,1, \dots, 0,1,2), \, L(\omega^{\prime}) = 2m, \, R_{n}(\widetilde\omega_{2i -1}) \leq 2^{n} + k, \,1 \leq i \leq m\right\}.\]

By (2.1), 
we see that
\[ P_{n + 1, +}^{u}\left(\left\{\omega : L(\omega^{\prime}) = 2m, \, R_{n + 1}(\omega) \leq  2^{n + 1} + k\right\}\right) \]
\[ = P_{1,+}^{u}\left(\left\{\zeta : \zeta = (0,1,0,1, \dots, 0,1,2), \, L(\zeta) = 2m\right\}\right) \cdot P_{n,+}^{u}(R_{n} \leq  2^{n} + k)^{m}. \]
\[ = u^{2m - 2}x_{u}^{2m - 1} P_{n,+}^{u}(R_{n} \leq  2^{n} + k)^{m}. \]

Then, 
\begin{align*}
 P_{n + 1,+}^{u}(R_{n + 1} \leq  2^{n + 1} + k) &=  \sum_{m = 1}^{\infty}u^{2m - 2}x_{u}^{2m - 1} P_{n,+}^{u}(R_{n} \leq  2^{n} + k)^{m} \\
 &=  \Phi\left(A_{u, 0} ; P_{n,+}^{u}({R_{n}} \leq  {2^{n}} + k)\right),
\end{align*}
which is the desired result.\\

Case 2. $2^{n} \leq k \leq 2^{n + 1} - 1$.

Since $L(\omega^{\prime}) = 2m$, 
we can write $\omega^{\prime} = (0,\epsilon_{1}, 0, \epsilon_{2}, \dots,0, \epsilon_{m - 1}, 0,1,2)$, $\epsilon_{i} \in \{ \pm 1\}$,  $1 \leq i \leq m-1$.
Then we see that 
\[ \left\{\omega : L(\omega^{\prime}) = 2m, \, R_{n + 1}(\omega) \leq  2^{n + 1} + k \right\} \]
\[ = \bigcup_{i = 0}^{m - 1} \left\{\omega : \sharp(j :\epsilon_{j} = -1) = i, \, L(\omega^{\prime}) = 2m, \, R_{n + 1}(\omega) \leq  2^{n + 1} + k \right\}. \]
We remark that the union in the above is disjoint.

For $1 \leq i \leq m - 1$, 
\[ \left\{ \omega :  \sharp(j : \epsilon_{j} = -1) = i, L(\omega^{\prime}) = 2m, R_{n + 1}(\omega) \leq  2^{n + 1} + k \right\} \]
\begin{multline*}
= \bigcup_{1 \leq n_{1} < n_{2} < \dots < n_{i} \leq m-1} \{\omega : \{ j :\epsilon_{j} = -1\} = \{n_{1} < n_{2} < \dots < n_{i}\}, \\
L(\omega^{\prime}) = 2m, R_{n + 1}(\omega) \leq  2^{n + 1} + k \}.
\end{multline*}
We remark that the union in the above is disjoint.

By (2.1),  
\begin{multline*}
P_{n + 1, +}^{u}(\{\omega : \{ j :\epsilon_{j} = -1\} = \{n_{1} < n_{2} < \dots < n_{i}\}, \\
L(\omega^{\prime}) = 2m, R_{n + 1}(\omega) \leq  2^{n + 1} + k  \}) 
\end{multline*} 
\begin{multline*}
= P_{n + 1, +}^{u}(\{ \omega :  \{ j :\epsilon_{j} = -1\} = \{n_{1} < n_{2} < \dots < n_{i}\}, \\
L(\omega^{\prime}) = 2m, R_{n}(\widetilde\omega_{2n_{j}} ) \leq  k, 1 \leq j \leq i \})
\end{multline*}
\[ = P_{1, +}^{u}\left(\left\{\omega^{\prime} : \{ j :\epsilon_{j} = -1\} = \{n_{1} < n_{2} < \dots < n_{i}\}, L(\omega^{\prime}) = 2m \right\}\right) P_{n,+}^{u}(R_{n} \leq  k)^{i} \]
\[ = u^{2m - 2}x_{u}^{2m - 1}\left(P_{n,+}^{u}(R_{n} \leq  k)\right)^{i}. \]

Since the number of choices $\{n_{1} < n_{2} < \cdots < n_{i}\} \subset \{1, \dots, m - 1\}$ is equal to $\dbinom{m - 1}{i}$,
we see that  
\[ P_{n + 1,+}^{u}(\{\omega :  \sharp(j :\epsilon_{j} = -1) = i, L(\omega^{\prime}) = 2m, R_{n + 1}(\omega) \leq  2^{n + 1} + k \}) \]
\begin{multline*}
= \sum_{1 \leq n_{1} < n_{2} < \dots < n_{i} \leq m-1} P_{n + 1,+}^{u}(\{\omega : \{ j :\epsilon_{j} = -1\} = \{n_{1} < n_{2} < \dots < n_{i}\}, \\
L(\omega^{\prime}) = 2m, R_{n + 1}(\omega) \leq  2^{n + 1} + k  \})
\end{multline*}
\[ =  \dbinom{m - 1}{i}u^{2m - 2}x_{u}^{2m - 1}\left(P_{n,+}^{u}(R_{n} \leq  k)\right)^{i}, \, \, 1 \leq i \leq m-1.\]
This is also true for $i = 0$.

Therefore, 
by summing up over $i$, we see that 
\[  P_{n + 1,+}^{u}\left(\left\{\omega : L(\omega^{\prime}) = 2m, R_{n + 1}(\omega) \leq  2^{n + 1} + k \right\}\right) 
 =  u^{2m - 2}x_{u}^{2m - 1} \left(1 + P_{n,+}^{u}(R_{n} \leq  k)\right)^{m - 1}.\]
By summing up over $m$, we see that 
\begin{align*} 
P_{n + 1,+}^{u}(R_{n + 1} \leq  2^{n + 1} + k) 
&=  \sum_{m = 1}^{\infty} u^{2m - 2}x_{u}^{2m - 1} \left(1 + P_{n,+}^{u}(R_{n} \leq  k)\right)^{m - 1}\\
&= \Phi\left(A_{u,1} ; P_{n,+}^{u}(R_{n} \leq  k)\right).
\end{align*}
This completes the proof.
\end{proof}

Next, we will show that $\tilde g_{u}$, which is the right continuous modification of $g_{u}$, satisfies (1.1) on $[0,1]$, not only on $D$. 
We define some notation.
Let $X_{n}(x) = \lfloor2^{n}x\rfloor-2\lfloor2^{n-1}x\rfloor$ and $\zeta_{n}(x) = \sum_{k =1}^{n} 2^{-k}X_{k}(x)$, $x \in [0,1)$, $n \geq 1$.
Then, $\zeta_{n}(x) \leq x < \zeta_{n}(x)+2^{-n}$, $x \in [0,1)$, $n \geq 1$.
Let $\gamma_{u} = 1/ \Phi(A_{u,0}; 1)$.
Let $p_{u,0}(z) = (z+1)/(z+\gamma_{u})$ and $p_{u,1}(z) = 1- p_{u,0}(z)$ for $z > -\gamma_{u}$.
Let \[ \begin{pmatrix} p_{u, n}(x) & q_{u, n}(x)  \\ r_{u, n}(x) & s_{u, n}(x) \end{pmatrix}  = A_{u, X_{1}(x)} \cdots A_{u, X_{n}(x)}, \, \, x \in [0,1), n \geq 1. \]

\begin{Prop}
$(1)$ $g_{u}(\zeta_{m}(x)) = \Phi(A_{u, X_{1}(x)} \cdots A_{u, X_{m}(x)}; 0)$ and \\
$g_{u}(\zeta_{m}(x) + 2^{-m}) = \Phi(A_{u, X_{1}(x)} \cdots A_{u, X_{m}(x)}; 1)$, $x \in [0,1)$, $m \geq 1$.\\
$(2)$ $\tilde g_{u} = g_{u}$ on $D$.\\
$(3)$ $\tilde g_{u}$ satisfies the equation $(1.1)$ on $[0,1]$. 
\end{Prop}

\begin{proof}
(1) Using $(1.1)$, we can show the assertion by induction in $n$. 

(2) By noting the definition of $g_{u}$ and $\tilde g_{u}$, 
we have that $\tilde g_{u}(1) = 1 = g_{u}(1)$.
Let $x \in D \cap [0,1)$.
Then, there exists $N$ such that $X_{n}(x) = 0$, $n > N$.

Then, by using the assertion (1), 
\begin{align*}
 \lim_{l \to \infty}g_{u}(x+ 2^{-l}) &= \lim_{l \to \infty} g_{u}(\zeta_{l}(x) + 2^{-l}) \\
 &= \lim_{m \to \infty}\Phi\left(A_{u, X_{1}(x)} \cdots A_{u, X_{N}(x)}; \Phi(A_{u,0}^{m};1)\right).
\end{align*}
Since $\Phi(A_{u,0}; \cdot)$ is a contraction map on $[0,1]$,
$\lim_{m \to \infty}\Phi(A_{u,0}^{m};1) = 0$.
Then, by using the assertion (1), 
\[ \lim_{m \to \infty}\Phi\left(A_{u, X_{1}(x)} \cdots A_{u, X_{N}(x)}; \Phi(A_{u,0}^{m};1)\right) = \Phi\left(A_{u, X_{1}(x)} \cdots A_{u, X_{N}(x)};0\right) = g_{u}(x).\]
Thus we obtain the assertion (2).

(3) Since $\tilde g_{u}(1) = 1$ and $\Phi(A_{u,1};1) = 1$,
(1.1) holds for $x = 1$. 

Let $x \in [0,1/2)$.
Then there exists a sequence $\{x_{n}\}_{n} \subset D \cap [0,1/2)$ such that $x_{n} \downarrow x$.
By using Proposition 3.3 and the assertion (2), 
$\tilde g_{u}(x_{n}) = \Phi\left(A_{u,0}; \tilde g_{u}(2x_{n})\right)$, $n \geq 1$.
Since $\Phi(A_{u,0}; \cdot)$ is continuous and $\tilde g_{u}$ is right continuous, 
we have that $\tilde g_{u}(x) = \Phi\left(A_{u,0}; \tilde g_{u}(2x)\right)$.

In the same manner, we see that $\tilde g_{u}(x) = \Phi\left(A_{u,1}; \tilde g_{u}(2x-1)\right)$ for $x \in [1/2, 1)$.
Thus we obtain the assertion (3).
\end{proof}

\begin{proof}[Proof of Theorem 1.2]
First, we show the assertion (1).
Let $\widetilde P_{n}^{u} = P^{u} \circ \left((R_{n}/2^{n}) - 1\right)^{-1}$. 
Let $\widetilde P^{u}$ be  the probability measure on $[0,1]$ whose distribution function is $\tilde g_{u}$ and satisfying $\widetilde P^{u}(\{0\}) = 0$.
In other words, we will show that the function $f_{u}$ in the statement in Theorem 1.2 is equal to $\tilde g_{u}$. 
It suffices to show that $\widetilde P_{n}^{u}$ converges weakly to $\widetilde P^{u}$, that is,
for any continuous function $f$ on $[0,1]$, 
\[ \lim_{n \to \infty} \int_{[0,1]} f(x) \widetilde P_{n}^{u}(dx) = \int_{[0,1]} f(x) \widetilde P^{u}(dx). \tag{3.1} \]

Let $\epsilon > 0$.
Then, 
$\max_{1 \leq k \leq  2^{m}} \left| f(k/2^{m}) - f((k - 1)/2^{m}) \right|  < \epsilon$ for some $m$. 
We have that
\begin{equation}
 \left| \int_{[0,1]} f(x) \widetilde P_{n}^{u}(dx) - \sum_{k = 1}^{2^{m}}  f\left(\frac{k}{2^{m}}\right) \widetilde P_{n}^{u}\left(\left[\frac{k - 1}{2^{m}}, \frac{k}{2^{m}}\right) \right)   \right|  < \epsilon ,\tag{3.2} \end{equation}
and, 
\begin{equation} \left| \int_{[0,1]} f(x) \widetilde P^{u}(dx) -  \sum_{k = 1}^{2^{m}}  f\left(\frac{k}{2^{m}}\right) \widetilde P^{u}\left(\left(\frac{k - 1}{2^{m}}, \frac{k}{2^{m}}\right] \right) \right|  < \epsilon, \tag{3.3}
\end{equation}
where we have used $\widetilde P_{n}^{u}(\{1\}) =  P^{u}(R_{n} = 2^{n + 1}) = P_{n,+}^{u}(R_{n} = 2^{n + 1}) = 0$ for the first inequality, 
and, $\widetilde P^{u}(\{0\}) = 0$ for the second.

Let $n > m$.
Then, by using Lemma 3.1,  we see that for $1 \leq k \leq 2^{m}$,
\[ \widetilde P_{n}^{u}\left(\left[\frac{k - 1}{2^{m}}, \frac{k}{2^{m}}\right) \right) = \widetilde P_{m}^{u}\left(\left[\frac{k - 1}{2^{m}}, \frac{k}{2^{m}}\right) \right) = g_{u}\left(\frac{k}{2^{m}}\right) -  g_{u}\left(\frac{k-1}{2^{m}}\right). \]

By using Proposition 3.4(2), 
we see that for $1 \leq k \leq 2^{m}$,
\[ \widetilde P^{u}\left(\left(\frac{k - 1}{2^{m}}, \frac{k}{2^{m}}\right] \right) = \tilde g_{u}\left(\frac{k}{2^{m}}\right) - \tilde g_{u}\left(\frac{k - 1}{2^{m}}\right) = g_{u}\left(\frac{k}{2^{m}}\right) -  g_{u}\left(\frac{k - 1}{2^{m}}\right).\]
Therefore, 
we see that 
\[ \sum_{k = 1}^{2^{m}}  f\left(\frac{k}{2^{m}}\right) \widetilde P_{n}^{u}\left(\left[\frac{k - 1}{2^{m}}, \frac{k}{2^{m}}\right) \right) = \sum_{k = 1}^{2^{m}}  f\left(\frac{k}{2^{m}}\right) \widetilde P^{u}\left(\left(\frac{k - 1}{2^{m}}, \frac{k}{2^{m}}\right] \right). \]
Recalling (3.2) and (3.3), 
we see that for any $n > m$, 
\[ \left| \int_{[0,1]} f(x) \widetilde P_{n}^{u}(dx) - \int_{[0,1]} f(x) \widetilde P^{u}(dx)\right|  < 2\epsilon. \]
Thus we see (3.1) and the proof of (1) completes.

The assertion (2) immediately follows from the definition of $\widetilde P^{u}$ and Proposition 3.4(3).

Finally, we show the assertion (3).
Let $u=1$.
Then, the absolute continuity of $\widetilde P^{1}$ follows from \cite{O2}, Theorem 1.2(1).

\begin{Lem}
Let $u \neq 1$.
Let $x \in [0,1] \setminus D$.
If $\tilde g_{u}$ is differentiable at $x$ and $\tilde g_{u}^{\prime}(x) \in [0,+\infty)$,
then, $\tilde g_{u}^{\prime}(x) = 0$.
\end{Lem}

\begin{proof}
We assume that there exists a point $x \in [0,1] \setminus D$ such that $\tilde g_{u}$ is differentiable at $x$ and $\tilde g_{u}^{\prime}(x) \in (0, +\infty)$.

Since $\tilde g_{u}$ is strictly increasing and $x \notin D$, 
we have that
\[ \tilde g_{u}^{\prime}(x) = \lim_{n \to \infty} 2^{n}(\tilde g_{u}(\zeta_{n}(x) + 2^{-n}) - \tilde g_{u}(\zeta_{n}(x))) = \lim_{n \to \infty} 2^{n}(g_{u}(\zeta_{n}(x) + 2^{-n}) - g_{u}(\zeta_{n}(x))).\]
Since $\tilde g_{u}^{\prime}(x) \in (0, +\infty)$, 
\[ \lim_{n \to \infty} \frac{g_{u}(\zeta_{n+1}(x) + 2^{-(n+1)}) - g_{u}(\zeta_{n+1}(x))}{g_{u}(\zeta_{n}(x) + 2^{-n}) - g_{u}(\zeta_{n}(x))} = \frac{1}{2}.\]

Then, by using Proposition 3.4(1), 
\[ p_{u, X_{n+1}(x)}\left(\frac{r_{u,n}(x)}{s_{u,n}(x)}\right) = \frac{g_{u}(\zeta_{n+1}(x) + 2^{-(n+1)}) - g_{u}(\zeta_{n+1}(x))}{g_{u}(\zeta_{n}(x) + 2^{-n}) - g_{u}(\zeta_{n}(x))}, \]
and, $\lim_{n \to \infty} p_{u, X_{n+1}(x)}(r_{u,n}(x)/s_{u,n}(x)) = 1/2$.
Since $p_{u,1} = 1-p_{u,0}$, \\
$\lim_{n \to \infty} p_{u, i}(r_{u,n}(x)/s_{u,n}(x)) = 1/2$ for $i=0,1$.
Now we see that \\
$\lim_{n \to \infty}r_{u,n}(x)/s_{u,n}(x) = \gamma_{u} - 2$.
Since $x \notin D$, 
there exists  infinitely many natural numbers $n$ such that $X_{n}(x) = i$ for each $i=0,1$.
Since $r_{u,n+1}(x)/s_{u,n+1}(x) = \Phi\left({}^t \! A_{u, X_{n+1}(x)}; r_{u,n}(x)/s_{u,n}(x)\right)$, 
we see that $\Phi\left({}^t \! A_{u,i}; \gamma_{u}-2\right) = \gamma_{u}-2$ for each $i=0,1$.
This is true if and only if $u=1$.
But this contradicts  the assumption. 
\end{proof}

Let $u \ne 1$.
Then, by noting Lemma 3.5 and the Lebesgue differentiation theorem, 
we see that $\tilde g_{u}^{\prime} = 0$ a.e. and  $\widetilde P^{u}$ is singular.
These complete the proof of (3). 
\end{proof}

\begin{proof}[Proof of Theorem 1.4]
In this proof, we write $\Phi_{u,i}(z) = \Phi(A_{u,i} ; z)$, $i = 0,1$.
We first explain the meaning of the value $u = \sqrt{3}$.
By explicit calculation,
we see that if $u < \sqrt{3}$, 
then, $0 < \Phi_{u,1}^{\prime}(z) < 1$, $z \in [0,1]$, 
namely,
$\Phi_{u,1}(\cdot)$ is a contraction map on $[0,1]$, and
de Rham's theory \cite{de} is applicable to $(A_{u,0}, A_{u,1})$ in the form of \cite{O2}. 
In contrast, this property fails if $u \ge \sqrt{3}$.
In fact, $\Phi_{\sqrt{3},1}^{\prime}(z) \le 1$, with $\Phi_{\sqrt{3},1}^{\prime}(z) = 1$ implying $z=1$.
If $u > \sqrt{3}$, there exists $z_{0} = z_{0}(u) \in (0,1)$ such that $\Phi_{u,1}^{\prime}(z) < 1$ for $z < z_{0}$, and  $\Phi_{u,1}^{\prime}(z) > 1$ for $z > z_{0}$. 

We now turn to the proof of the theorem. 
We denote $f^{m+1} = f \circ f^{m}$, $m \geq 1$, for $f : [0,1] \to [0,1]$. 

(1)  If $0 < u < \sqrt{3}$, then, $(A_{u,0}, A_{u,1})$ satisfies the conditions (A1) - (A3) in \cite{O2} and hence $\widetilde P^{u}$ has no atoms. 

Let $u = \sqrt{3}$. 
Let $h_{i} = \Phi_{ \sqrt{3},i}$, $i=0,1$.
Then we have the following results by computations.
\begin{Lem}
$(1)$ $h_{0}(z) < h_{1}(z)$ for $z \in [0,1]$. \\
$(2)$ $h_{i}^{\prime}$, $i=0,1$, are strictly increasing on $(0,1)$.\\
$(3)$ $h_{0}^{\prime}(z) \leq 3h_{1}^{\prime}(z)$ for $z \in (0,1)$.\\
$(4)$ $h_{0}^{\prime}(z) \leq h_{1}^{\prime}(z)$ for $z \geq h_{1}^{2}(0)$.
\end{Lem}

Now it is sufficient to show the following. 
\[ \lim_{m \to \infty}\max_{1 \leq k \leq 2^{m}} \left\{g_{\sqrt{3}}\left(\frac{k}{2^{m}}\right) - g_{\sqrt{3}}\left(\frac{k - 1}{2^{m}}\right) \right\} = 0. \tag{3.4} \]

Let $m \geq 3$ and $1 \leq k \leq 2^{m}$.
Let $x_{i} = X_{i}((k - 1)/2^{m})$, $1 \leq i \leq m$.
Then, 
\begin{align*}
g_{\sqrt{3}}\left(\frac{k}{2^{m}}\right) - g_{\sqrt{3}}\left(\frac{k - 1}{2^{m}}\right)  &= h_{x_{1}} \circ \cdots \circ h_{x_{m}}(1) - h_{x_{1}} \circ \cdots \circ h_{x_{m}}(0) \\
&= \int_{0}^{1} (h_{x_{1}} \circ \cdots \circ h_{x_{m}})^{\prime}(x) dx \\
&=  \int_{0}^{1} h_{x_{1}}^{\prime}(h_{x_{2}} \circ \cdots \circ h_{x_{m}}(x))  \cdots h_{x_{m - 1}}^{\prime}(h_{x_{m}}(x)) h_{x_{m}}^{\prime}(x)  dx \\
&\leq  \int_{0}^{1} h_{x_{1}}^{\prime}(h_{1}^{m-1}(x))  \cdots h_{x_{m - 1}}^{\prime}(h_{1}(x)) h_{x_{m}}^{\prime}(x)  dx \\
&\leq  \int_{0}^{1} h_{1}^{\prime}(h_{1}^{m-1}(x))  \cdots 3 h_{1}^{\prime}(h_{1}(x)) 3 h_{1}^{\prime}(x)  dx \\
&= 9 \int_{0}^{1} (h_{1}^{m})^{\prime}(x) dx  = 9(1- h_{1}^{m}(0)),
\end{align*}
where we have used Proposition 3.4 (1) for the first equality, Lemma 3.6 (1) and (2) for the fourth inequality, and,  Lemma 3.6 (3) and (4) for the fifth.
Since $h_{1}^{n}(0) = n/(n + 1)$, $n \geq 1$, 
we see that $\lim_{n \to \infty} h_{1}^{n}(0) = 1$.
Thus we see (3.4) and 
the proof of the assertion (1) completes. 

(2) 
Let $x \in D \cap (0,1)$.
Let $x_{i} = X_{i}(x)$, $i \geq 1$.
Then, there exists a unique $m \geq 1$
such that $x_{m} = 1$ and $x_{i} = 0$, $i \geq m+1$.
Let $\phi = \Phi_{u, x_{1}} \circ \cdots \circ \Phi_{u, x_{m - 1}} \circ \Phi_{u,0}$.
Let $n > m$ and $y_{i} = X_{i}(x - (1/2^{n}))$.
Then, we have that $y_{i} = x_{i}$, $1 \leq i \leq m-1$, 
$y_{m}  = 0$, 
$y_{i}  = 1$, $m+1 \leq i \leq  n$, and, 
$y_{i}  = 0$, $i >  n$. 
By noting Proposition 3.4 (1) and $\Phi_{u,0}(1) = \Phi_{u,1}(0)$, we have that 
\[g_{u}(x) = \phi (1) , \, \, g_{u}\left(x - \frac{1}{2^{n}}\right) = \phi (\Phi_{u, 1}^{n-m}(0)). \tag{3.5} \]

Note that $\Phi_{u,1}$ is increasing and strictly convex, $\Phi_{u,1}(0) > 0$, $\Phi_{u,1}(1) = 1$, and, $\Phi_{u,1}^{\prime}(1) > 1$. 
Therefore, 
there exists $z_{1} \in (0,1)$ such that
\[ \Phi_{u, 1}(z_{1}) = z_{1}, \, \Phi_{u, 1}(z) > z, \, z \in (0, z_{1}), \,  \Phi_{u, 1}(z) < z, \, z \in (z_{1}, 1).\]
Then, $z_{1} = \lim_{n \to \infty} \Phi_{u, 1}^{n}(0)$ and $\Phi_{u, 1}^{n}(0) \leq z_{1} < 1$, $n \ge 1$.

We have that for $n > m$, 
\begin{align*}
	\widetilde P^{u}\left(\left(x - \frac{1}{2^{n}}, x\right]\right) &=  g_{u}(x) -  g_{u}\left(x - \frac{1}{2^{n}}\right)  \\
									  									  &=  \phi (1) -  \phi (\Phi_{u, 1}^{n-m}(0)) \\
									  									&\ge \phi (1) - \phi ( z_{1}),
\end{align*}
where we have used Proposition 3.4 (2) for the first equality, and, (3.5) for the second.
Letting $n \to \infty$, we have that $\widetilde P^{u}(\{x\}) \ge \phi(1)  - \phi (z_{1}) > 0$.

We can show that $\widetilde P^{u}(\{1\}) > 0$ in the same manner.
These complete the proof of the assertion (2).
\end{proof}

\vspace{1\baselineskip}

\begin{flushright}
\textsc{Graduate School of Mathematical Sciences, \\
The University of Tokyo \\
Komaba 3-8-1, Meguro-ku, Tokyo, 153-8914, Japan \\ 
e-mail address :} \texttt{kazukio@ms.u-tokyo.ac.jp}
\end{flushright}

\end{document}